\documentclass[11pt]{article}

\usepackage{amsmath,amsfonts,amsthm,amscd,amssymb,graphicx}
\numberwithin{equation}{section}

\usepackage{color}

\newtheorem{theorem}{Theorem}[section]

\newtheorem{lemma}[theorem]{Lemma}

\newtheorem{proposition}[theorem]{Proposition}

\newtheorem{remark}[theorem]{Remark}
\newtheorem{definition}[theorem]{Definition}

\begin{document}

\title{Instabilities  in the mean field limit
}

\author{Daniel Han-Kwan\footnotemark[1] \and Toan T. Nguyen\footnotemark[2]}

\maketitle

\renewcommand{\thefootnote}{\fnsymbol{footnote}}

\footnotetext[1]{CNRS $\&$ \'Ecole polytechnique, Centre de Math\'ematiques Laurent Schwartz UMR 7640, 91128 Palaiseau Cedex, France. Email: daniel.han-kwan@polytechnique.edu}
\footnotetext[2]{Department of Mathematics, Pennsylvania State University, State College, PA 16802, USA. Email: nguyen@math.psu.edu. TN was supported in part by the NSF under grant DMS-1405728.}

\begin{abstract}
Consider a system of $N$ particles interacting through Newton's second law with Coulomb interaction potential in one spatial dimension or a $\mathcal{C}^2$ smooth potential in any dimension. We prove that in the mean field limit $N \to + \infty$, the $N$ particles system displays instabilities in times of order $\log N$ for some configurations approximately distributed according to unstable homogeneous equilibria.
\end{abstract}

\section{Introduction}

Consider an $N$-particle system described by position and velocity $(X_{k,N}(t), V_{k,N}(t))_{1\leq k \leq N}$ in the 
phase space $\mathbb{T}^d \times \mathbb{R}^d$, $d \geq 1$. Assume that their dynamics follows the Newton's second law:
\begin{equation}
\begin{aligned}
\frac{d}{dt} &X_{k,N} = V_{k,N} , \qquad \frac{d}{dt} V_{k,N} = - \frac 1N \sum_{j \not = k} \nabla_x \Phi (X_{k,N} - X_{j,N}), \\
&{X_{k,N}}_{|t=0}=X_{k,N}^0,  \qquad {V_{k,N}}_{|t=0}=V_{k,N}^0,
\end{aligned}
\end{equation}
in which $\Phi \neq 0$ denotes an interaction potential. To enforce that the law of action-reaction is satisfied, we impose that $\Phi$ satisfies $\Phi(\cdot) = \Phi(-\cdot)$.
In this work we restrict to two classes of potentials:
\begin{itemize}
\item {\bf A1.} The case of smooth potentials $\Phi$, that is $\Phi \in \mathcal{C}^2$, in any dimension $d\in \mathbb{N}^*$; for instance, this includes the Hamiltonian Mean Field (HMF) model.

\item  {\bf A2.} The case of  Coulomb potential for $d=1$. Identifying $\mathbb{T}$ to the interval $[-1/2,1/2)$, the interaction potential in this case reads
$$
\Phi(x) = \frac{|x|-x^2}{2}.
$$

\end{itemize}

In the case of {\bf A1}, the existence and uniqueness for all times of the trajectories follow from the Cauchy-Lipschitz theorem. In the case of {\bf A2}, one can observe that the potential has one singularity at $x=0$, that corresponds to the situation when two particles collide. Hauray showed in \cite{HauX} how to use the theory of differential inclusions to obtain existence and uniqueness of trajectories for all initial conditions.

Consider now the so-called empirical measure
$$ \mu_N(t): = \frac 1N \sum_{k=1}^N \delta_{(X_{k,N}(t), V_{k,N}(t))}.$$
Then, in the distributional sense, the measure $\mu_N(t)$ solves the one particle Liouville equation: 
$$ \partial_t \mu_N + v \cdot \nabla_x \mu_N + E_N \cdot \nabla_v \mu_N = 0, \qquad E_N (t,X_{k,N}(t))=  - \frac 1N \sum_{j \not =k} \nabla_x \Phi (X_{k,N}(t) - X_{j,N}(t)),$$
posed on the one particle phase space $\mathbb{T}^d \times \mathbb{R}^d$.  
Formally, as $N\to \infty$, assuming that  $\mu_N \to f$ in some appropriate sense, one gets the nonlinear Vlasov system
\begin{equation}
\label{Vlasov} \partial_t f + v \cdot \nabla_x f + E \cdot \nabla_v f = 0, \qquad E = - \iint_{\mathbb{T}^d \times \mathbb{R}^d} \nabla_x \Phi(x-y) f(t,y,v)\; dy dv.
\end{equation}
In the case where the potential is coulombian, one obtains the classical Vlasov-Poisson system.
Before going any further, let us now recall the well-posedness results which are known for \eqref{Vlasov}.
Denote by $\mathcal{P}_1$ the set of Borel probability  measures on $\mathbb{T}^d \times \mathbb{R}^d $ with finite first moment.

\begin{theorem}

In the case of {\bf A1}, for any distribution $\nu_0 \in \mathcal{P}^1$, there is a unique global weak solution $\nu(t) \in C([0,+\infty), \mathcal{P}^1)$ with initial condition $\nu_0$ to \eqref{Vlasov}.

In the case of {\bf A2}, for any distribution $\nu_0 \in \mathcal{P}^1$, there is a global weak solution $\nu(t) \in C([0,+\infty), \mathcal{P}^1)$ with initial condition $\nu_0$ to \eqref{Vlasov}. Moreover, if $\nu_0$ generates a solution $\nu(t)$ such that for $T>0$, $\int_0^T \| \int \nu(dv) \|_{L^\infty_x} \, ds < +\infty$ then this solution is unique on $[0,T]$.

\end{theorem}

In the case of {\bf A2}, the existence is due to Zheng and Majda \cite{ZM} and Hauray \cite{HauX}, and weak-strong uniqueness was obtained by Hauray \cite{HauX} (see Theorem \ref{Hau} below).


\bigskip

As usual we will say that the mean field limit holds if for any sequence of initial configurations $(X_{k,N}(t), V_{k,N}(t))_{1\leq k \leq N, N \geq 1}$ such that 
$
\mu_{N}(0) 
$
weakly-$\star$ converges in the sense of measures to some (regular) initial distribution function $f_0(x,v)$, then there is a $T>0$ such that, for all $t \in [0,T]$, 
$
\mu_N(t)
$
weakly-$\star$ converges in the sense of measures to a distribution function $f(t,x,v)$ which is a solution of \eqref{Vlasov} with initial condition $f_0$.

We refer to the lecture notes of Golse \cite{Golse} and to the recent review of Jabin about this topic \cite{Jab}.
%
The case of smooth potentials in any dimension is very well-understood; see
Braun and Hepp \cite{BH}, Neunzert and Wick \cite{NW}, and Dobrushin \cite{Dob}.
The celebrated work of Dobrushin provides quantitative estimates for the mean field limit, that we shall recall below.

First let us introduce the classical $W_1$ Monge-Kantorovich distance. 
\begin{definition}
For any Borel probability  measures $\mu, \nu \in \mathcal{P}_1$,
we define 
$$
W_1(\mu, \nu) = \inf_{\pi \in \Pi(\mu,\nu)} \int_{\mathbb{T}^d \times \mathbb{R}^d} \int_{\mathbb{T}^d \times \mathbb{R}^d}  |z_1-z_2| \pi(dz_1, dz_2),
$$
where $\Pi(\mu, \nu)$ is the set of Borel probability measures on $[\mathbb{T}^d \times \mathbb{R}^d]^2 $ with first marginal $\mu$ and second marginal $\nu$.
\end{definition}
We recall that this notion allows to metrize on $\mathcal{P}_1$ the weak-$\star$ convergence in the sense of measures. We are now in position to state Dobrushin's estimate for weak solutions of \eqref{Vlasov}.
\begin{theorem}[Dobrushin]
\label{Dob}
Consider case {\bf A1}. For all probability measures $\mu_0, \nu_0 \in \mathcal{P}_1$, denoting by $\mu(t)$ (resp. $\nu(t)$) the unique solution of \eqref{Vlasov} with initial condition $\mu_0$ (resp. $\nu_0$), we have for all $t \geq 0$,
\begin{equation}\label{bd-Dob}
W_1(\mu(t), \nu(t)) \leq e^{C_0 t} W_1 (\mu_0, \nu_0),
\end{equation}
in which $C_0 = 2\| \nabla^2 \Phi\|_{\infty}$.  

\end{theorem}

One can apply this estimate for the mean field limit and deduce that as soon as $W_1 (\mu_N(0), f_0)\to 0$, then the following convergence also holds 
$$
W_1 (\mu_N(t), f(t) )\to 0,
$$
as $N \to \infty$, for all finite times $t \ge 0$. 
More precisely,  
if $W_1 (\mu_N(0), f_0) \lesssim \frac{1}{N^s}$ for some $s>0$, then the mean field limit holds, within times of order $\log N$, as $N \to \infty$.

For what concerns the case of singular interaction potentials, including the Coulombian case, things are much less understood. However, in the one-dimensional case, the Coulombian potential is only weakly singular, and it turns out that the mean field limit can be justified. This was performed by Trocheris \cite{Tro}, see also Cullen, Gangbo and Pisante \cite{CGP}.
Later, Hauray gave a remarkable proof of this mean field limit by exhibiting a stability estimate in the same spirit as Dobrushin's one (see \cite[Theorem 1.9]{HauX}).
Let us recall this result below.

\begin{theorem}[Hauray]
\label{Hau}
Consider case {\bf A2}. Let  $f_0(x,v)$  be such that the corresponding solution $f(t)$ of \eqref{Vlasov} satisfies for some $T>0$,
$$
\sup_{t \in [0,T]} \| \rho(t) \|_\infty < +\infty,
$$
denoting $\rho(t) = \int f \, dv $. Then, there is $C_1>0$ such that for any $ \mu_0 \in \mathcal{P}_1$, denoting by $\mu(t)$ any weak solution of \eqref{Vlasov} with initial condition $\mu_0$, there holds for all $t \in [0,T]$,
$$
W_1(\mu(t), f(t)) \leq e^{C_1 (t+  \int_0^t \| \rho(s) \|_\infty \, ds ) } W_1 (\mu_0, f_0).
$$
\end{theorem}
It is known (see for instance \cite[Proposition 1.10]{HauX} and references therein) that when $f_0$ is smooth and decaying, then the $L^\infty$ norm of the associated density is bounded for all times. Arguing exactly as we did with Dobrushin's estimate, this theorem allows to prove the mean field limit for all finite times (independent of $N$) and for times of order $\log N$  when $W_1 (\mu_N(0), f_0)$ decays at a polynomial rate.

\bigskip

In higher dimension, only partial results are available; we shall not discuss this difficult topic and only refer to Hauray-Jabin \cite{HJ1,HJ2} and to very recent works of Lazarovici-Pickl \cite{LP}, Lazarovici \cite{L} (see also Barr\'e, Hauray and Jabin \cite{BHJ}).

\bigskip

In this paper we investigate  the large time (in terms of $N$) mean field limit.
To the best of our knowledge, the only mathematical works in this direction are due to Caglioti and Rousset \cite{CR1,CR2};
in the case of smooth potentials, they obtain long time estimates in the \emph{stable} situation, that is when the initial empirical measure converges to a \emph{stable} homogeneous equilibrium. Loosely speaking, the papers \cite{CR1,CR2} justify the mean field limit for polynomially growing times, that is times of order $N^\alpha$, for some  $\alpha>0$.



We focus in this work on the role of \emph{unstable} homogeneous equilibria. We shall consider \emph{smooth}  homogeneous equilibria $f_\infty$, that is, $f_\infty \in C^k(\mathbb{R}^d)$, with $k \gg 1$ and \emph{decaying sufficiently fast at infinity}\footnote{Note that the precise required smoothness and decay can be quantified from an inspection of our analysis.}, such that:
\begin{itemize}
\item $f_\infty$ is \emph{radial}, that is to say $f_\infty \equiv f_\infty(|v|)$;
\item $f_\infty$ is \emph{normalized}, in the sense that $\int_{\mathbb{R}^d} f_\infty(v) \, dv =1$;
\item $f_\infty(v)$ is positive and satisfies
\begin{equation} \label{e-nonnegativity}
\sup_{v \in \mathbb{R}^d}  \,  \frac{|\nabla f_\infty(v)|}{(1+ |v|)f_\infty(v)} < + \infty.
\end{equation}
\end{itemize}
The last condition is useful to ensure that the distribution functions we will consider are non-negative.

Let $(\widehat \Phi_k)_{k  \in \mathbb{Z}^d}$ be the Fourier coefficients of $\Phi$. Following \cite{Pen,GS}, we introduce the following definition.
\begin{definition}
\label{DefPen}

The homogeneous equilibrium $f_\infty$ satisfies the Penrose instability condition if there is a $k_0 = (k_{0,1},0\cdots,0) \neq 0$ such that
$\widehat \Phi_{k_{0}}>0$ and
\begin{equation}\label{cond-Pen}
\int_{\mathbb{R}^d} \frac{f_\infty(v) - f_\infty(0,v_2,\cdots,v_d)}{ v_1^2} \, dv >  \frac{1}{\widehat \Phi_{k_{0}}}.
\end{equation}
\end{definition}
We shall explain in Section \ref{SecInsta} why such a condition yields instability. Of course, many variants of Definition~\ref{DefPen} can be considered, but we shall stick to this one for simplicity.
 The assumption that for at least one $k_{0,1} \neq 0$, $\hat \Phi_{(k_{0,1},0\cdots,0)} >0$, is verified automatically for the HMF model and Coulomb potential in dimension one. In the case of {\bf A1}, we will systematically assume that there is at least one mode $k_0  = (k_{0,1},0,\cdots,0)\neq 0$ such that 
$\widehat \Phi_{k_0} > 0$.

%
%
%

Our main result  shows that given an equilibrium $f_\infty$ satisfying the Penrose instability condition, we can find initial configurations such that $W_1(\mu_N(0), f_\infty(v))$ converges polynomially fast to $0$ as $N \to + \infty$, but such that $W_1(\mu_N(t), f_\infty(v))$ does not go to $0$ for times of order $\log N$.
This means that the times of order $\log N$ reached via Theorem \ref{Dob} or \ref{Hau} are optimal.

\begin{theorem}
\label{thm1}
Consider {\bf A1} or {\bf A2}.
Let $f_\infty (v)$ be a smooth equilibrium satisfying the Penrose instability condition of Definition \ref{DefPen}. There is $\alpha_0>0$ such that for any $\alpha \in (0, \alpha_0]$, there exists a sequence of initial configurations $(X_{k,N}^0, V_{k,N}^0)_{1\leq k \leq N, \, N \geq 1}$ 
such that as $N \to \infty$, 
$$
W_1 (\mu_{N}(0), f_\infty) \sim \frac{1}{N^\alpha},
$$
where $ \mu_N(0): = \frac 1N \sum_{k=1}^N \delta_{(X_{k,N}^0, V_{k,N}^0)}$
and
$$
\limsup_{N\to +\infty}  W_1 (\mu_{N}(T_N), f_\infty) >0,
$$
with $T_N = O(\log N)$.

\end{theorem}

Let us note that Jain, Bouchet and Mukamel already exposed  in \cite{JBM}  the mechanism of linear instability on which Theorem~\ref{thm1} is based; however \cite{JBM} did not treat the nonlinear analysis.
The basic idea to obtain such a result is to rely on the instability of the linearized Vlasov equations around unstable equilibria $f_\infty$. As will be recalled below, one can get an eigenvalue $\lambda_0$ with maximal real part for this problem. For simplicity let us assume in this discussion that $\lambda_0$ is real. Let $g(x,v)$ be an associated eigenfunction. This means that there exists a solution to the linearized equations around $f_\infty$
which grows like $\varepsilon e^{\lambda_0 t}$, $\varepsilon(N)>0$ being the size of the perturbation of the initial time, that may depend on $N$. Then one can generate a solution $f(t)$ to the nonlinear Vlasov equations displaying such a behavior, but with an error term of order 
$\varepsilon^2 e^{2\lambda_0 t}$, coming from the quadratic error made when forgetting about the nonlinear terms. Let now $\lambda>0$ be the best constant one can achieve in the exponential in time for the Dobrushin or Hauray stability estimates. One may end up with estimates of the form

$$
\begin{aligned}
W_1(\mu_N(t), f_\infty) &\geq W_1(f(t), f_\infty) -W_1(\mu_N(t), f(t)) \\
&\geq \varepsilon e^{\lambda_0 t} -  \varepsilon^2 e^{2\lambda_0 t} -W_1(\mu_N(t), f(t))
\end{aligned}
$$
and
$$
W_1(\mu_N(t), f(t)) \lesssim e^{\lambda t} W_1 (\mu_{N,0}, f(0)).
$$
One can prove that such estimates are useful to prove instability, up to taking $\varepsilon$ appropriately.  
However, to prove that the aforementioned error is indeed at least quadratic, one needs that $\lambda_0$ is large enough, which is virtually uncheckable in practice.

Our main contribution is to explain how to overcome this problem, by using Grenier's iterative scheme for proving nonlinear instability results (\cite{G}), that allows to build a high order approximation of the growing solution of the linearized equation.


\begin{remark}

One can obtain from the proof upper bounds on the parameter $\alpha_0$ in the statement of Theorem~\ref{thm1}. For instance for the case of smooth potentials in $d=3$, directly from \eqref{parameters} with $s = \frac16$ and $C_2=C_0$ as in \eqref{bd-Dob}, we get
$$
\alpha_0 \leq \frac{1}{6} \frac{1}{\lceil\frac{2\| \nabla^2 \Phi\|_{\infty}}{\Re \lambda_0}\rceil},
$$ 
where $\lambda_0$ is an eigenvalue with maximal real part. 

One can see from the proof that that the fact that
$$
W_1 (\mu_{N}(0), f_\infty) \sim \frac{1}{N^\alpha}
$$
(with $\alpha \leq \alpha_0$) is crucial in our argument,
and we do not know what happens if we only look at initial configurations that converge very fast (i.e. much faster than the rate $\frac{1}{N^\alpha}$) to the equilibrium.

\end{remark}

The following of the paper is devoted to the proof of Theorem \ref{thm1}.

\section{Spectral analysis}
\label{SecInsta}

We start by an analysis of the spectral instability for the linearized Vlasov equations around Penrose unstable equilibria.

Let us denote by $L$ the linearized Vlasov operator around a homogenous equilibrium state $f_\infty$; namely, 
\begin{equation}\label{def-linL} L f : =  - v \cdot \nabla_x f  -  E[f] \cdot \nabla_v f_\infty, \qquad E[f] (x) = -  \iint_{\mathbb{T}^d\times \mathbb{R}^d} \nabla \Phi(x-y) f(y,v)\; dy dv.\end{equation}
For convenience, we set $\rho = \int_{\mathbb{R}^d} f\; dv$. It is clear that 
$$E[f] = -\nabla_x (\Phi \star \rho) =-( \nabla_x \Phi) \star \rho,$$ 
in which $\star$ denotes the usual convolution with respect to the variable $x$. Since $E[f]$ involves no derivative of $f$ and $f_\infty(v)$ decays rapidly as $v\to \infty$, $L$ is a compact perturbation of $-v \cdot \nabla_x f$, whose continuous spectrum in $\langle v \rangle^m$-weighted $L^2$ spaces lies on the imaginary axis, for $m > \frac d2$. Hence, the possible unstable spectrum of $L$ consists of eigenvalues $\lambda$ solving $(\lambda - L) f = 0$. To analyze the point spectrum, we project the eigenvalue equation on each Fourier mode, that is we take the scalar product $\langle e^{ik\cdot x}, \cdot \rangle_{L^2(\mathbb{T}^d)}$ for $k \in \mathbb{Z}^d$, yielding
\begin{equation}
\label{instable}
 (\lambda  + i k \cdot v) \hat f_k - ik \cdot \nabla_v f_\infty (v)  \widehat\Phi_k  \hat \rho_k = 0,
 \end{equation}
in which $\hat f_k, \widehat \Phi_k,$ and $\hat  \rho_k$ denote the Fourier coefficients of $f, \Phi, $ and $\rho$, respectively. The above yields 
$$ \Big [ 1 - \widehat\Phi_k \int_{\mathbb{R}^d} \frac{ik \cdot \nabla_v f_\infty (v) }{\lambda + i k \cdot v} \; dv \Big] \hat \rho_k = 0.$$
This proves that if $\lambda$ is an unstable eigenvalue of $L$ with the corresponding eigenfunction $f $, then the following must hold
\begin{equation}\label{Penrose}
 \widehat\Phi_k \int \frac{ik \cdot \nabla_v f_\infty (v)}{ \lambda+ ik \cdot v } \; dv = 1, \qquad \Re \lambda >0,
 \end{equation}
 for all $k$ so that $\hat \rho_k \not =0$. In addition, we note that for both Coulomb and $\mathcal{C}^2$ smooth potentials, $| k \hat \Phi_k| \to 0$ as $|k|\to \infty$. This proves that there are finite values of $k \in \mathbb{Z}^d$ so that $\hat \rho_k$, and so $\hat f_k$, are non zero solutions of \eqref{instable}. That is, if $\lambda$ is an unstable eigenvalue, then the corresponding eigenfunctions $f$ are a linear combination of finite modes $e^{ik\cdot x} f_k$, with 
 \begin{equation}
 \label{eigen}
  f_k =  \widehat\Phi_k \frac{ik \cdot \nabla_v f_\infty (v)}{ \lambda+ ik \cdot v }  . 
 \end{equation}
In particular, the $v$-regularity and $v$-decaying property of the eigenfunctions are almost the same as those of $f_\infty(v)$. Finally, we remark that the left hand side in \eqref{Penrose} tends to zero, as $|\lambda|\to \infty$. This shows that the unstable point spectrum lies in a bounded positive semi-circle in $\mathbb{C}$, and so a maximal unstable eigenvalue exists, if the unstable spectrum is non empty.  

Let us now explain why the Penrose instability condition, recall Definition~\ref{DefPen}, implies the existence of a growing mode. Indeed, let $f_\infty(v)$ be an equilibrium satisfying the condition \eqref{cond-Pen}.  
We note that $$
\lim_{\lambda \to +\infty} \int_{\mathbb{R}^d} \frac{v_1 \partial_{v_1} f_\infty(v)}{\lambda^2 + k_0^2 v_1^2} \, dv =0,
$$
and that the Penrose instability condition \eqref{cond-Pen} reads
$$
\int_{\mathbb{R}^d} \frac{f_\infty(v) - f_\infty(0,v_2,\cdots,v_d)}{ k_0^2 v_1^2} \, dv  = \int_{\mathbb{R}^d} \frac{v_1 \partial_{v_1} f_\infty(v)}{ k_0^2 v_1^2} \, dv > \frac{1}{k_0^2 \, \widehat \Phi_{k_0}} .
$$
It follows by a continuation argument that there exists a $\lambda>0$ such that
$$
 \int_{\mathbb{R}^d} \frac{v_1 \partial_{v_1} f_\infty(v)}{\lambda^2 + k_0^2 v_1^2} \, dv =\frac{1}{k_0^2 \,  \widehat \Phi_{k_0}}.
$$
Now, following Guo-Strauss \cite{GS}, define
$$
g(x,v) =  \frac{v_1 \partial_{v_1} f_\infty(v)}{\lambda^2 + k_0^2 v_1^2} k_0  \cos (k_0 x_1) - \frac{  \partial_{v_1} f_\infty(v)}{\lambda^2 + k_0^2 v_1^2} \lambda \sin(k_0 x_1).
$$
It can be checked by direct computations that
$$
\begin{aligned}
\lambda g + v\cdot \nabla_x g
&=  \frac{  \partial_{v_1} f_\infty(v)}{\lambda^2 + k_0^2 v_1^2} \left(   \lambda k_0 v_1 \cos(k_0 x_1)  - \lambda^2 \sin(k_0 x_1) - \lambda k_0 v_1 \cos(k_0 x_1) - v_1^2 k_0^2 \sin (k_0 x_1)   \right) \\
&=  - \partial_{v_1}f_\infty(v) \sin (k_0 x_1)
\\E[g] \cdot \nabla_v f_\infty  &= \partial_{v_1}f_\infty (v)\sin (k_0 x_1),
\end{aligned}
$$
so that $g(x,v)$ is an eigenfunction associated to the eigenvalue $\lambda>0$.

We gather the results of this section in the following statement.

\begin{proposition}
Let $f_\infty$ be a smooth homogeneous equilibrium satisfying the Penrose instability condition of Definition~\ref{DefPen}. The linearized Vlasov equations around $f_\infty$ admit at least one unstable eigenvalue.

Moreover, there exists an unstable eigenvalue $\lambda_0$ with maximal real part $\Re \lambda_0>0$. The corresponding eigenfunctions are analytic in $x$ and are smooth and fast decaying in $v$.
\end{proposition}

\section{Vlasov high order approximate solutions} 

In this section, we construct a good approximate solution to the Vlasov equation \eqref{Vlasov}. 
Our construction follows the iterative scheme introduced by Grenier \cite{G} (see also \cite{HKH, HKN} for recent applications in kinetic theory). 
Assume for simplicity that there is a real maximum eigenvalue $\lambda_0$ (we shall explain how to handle the general case at the end of the section). Let $g_1 = e^{\lambda_0 t} \hat g_1$ be an associated maximal growing mode of the linearized Vlasov problem: $(\partial_t - L) g_1 = 0$. Fix an integer $K\geq 1$. We set 
\begin{equation}\label{def-fapp} f_\mathrm{app}: = f_\infty + \sum_{k=1}^K \varepsilon^k   g_k \end{equation}
for sufficiently small $\varepsilon>0$, in which the profiles $g_k$ inductively solve the linear problem
\begin{equation}
\label{def-gk}
(\partial_t - L) g_k =  - \sum_{j = 1}^{k-1} E[g_{k-j}] \cdot \nabla_v g_{j}
\end{equation}
with zero initial data for $g_k$, $k\ge 2$. 
Then, $f_\mathrm{app}$ approximately solves the nonlinear Vlasov equation \eqref{Vlasov} in the sense
$$ \partial_t f_\mathrm{app} + v \cdot \nabla_x f _\mathrm{app}+ E[f_\mathrm{app}] \cdot \nabla_v f_\mathrm{app} = R_\mathrm{app}: = - \sum_{1\le j,\ell \le K; j+\ell \ge K+1} \varepsilon^{j+k}E[g_j] \cdot \nabla_v g_{\ell} .$$
We shall now show that $R_\mathrm{app}$ is indeed arbitrarily small, within a time interval of order $|\log \varepsilon|$, as $\varepsilon \to 0$. Precisely, we obtain the following proposition. 

\begin{proposition}\label{prop-Vapp} Let $m_0 > \frac d2$ and $ s\ge 0$ be fixed, and let $\Phi$ be the Coulomb or a $\mathcal{C}^2$ smooth potential, let $f_\infty(v)$ be a smooth unstable equilibrium which decays sufficiently fast as $v \to \infty$, and let $\lambda_0$ be a maximal unstable eigenvalue of $L$. Then, there are positive numbers $K_0, \varepsilon_0$, so that for sufficiently small $\varepsilon>0$, $K \ge K_0$, and for $T_\varepsilon = \frac{1}{\Re \lambda_0} \log \varepsilon_0/\varepsilon$, the Vlasov equation \eqref{Vlasov} has an unique solution $f\in C(0,T_\varepsilon; H^s)$ with initial data $f_0 = f_\infty + \varepsilon \hat g_1$. In addition, the following estimates holds:
\begin{equation}\label{diff} \| \langle v\rangle^{m_0}  (f - f_\mathrm{app})(t)\|_{H^s} \le C_{K} \Big( \varepsilon  e^{\Re \lambda_0 t}\Big)^{K},\qquad \forall t\in [0,T_\varepsilon],\end{equation}
\begin{equation}\label{diff2} \| \langle v \rangle^{m_0} R_\mathrm{app}\|_{H^s} \le C_{K} \Big( \varepsilon  e^{\Re \lambda_0 t}\Big)^{K+1},\qquad \forall t\in [0,T_\varepsilon].\end{equation}
Here, $C_{K}$ might depend on $K$, but is independent of $\varepsilon$, and $f_\mathrm{app}$ is defined as in \eqref{def-fapp}. Furthermore, the profiles $g_k$ satisfy $\|  \langle v \rangle^{m_0} g_k(t) \|_{H^s} \le  C_0 e^{k\Re \lambda_0 t}$, for all $k\ge 1$ and $t \ge 0$. 
\end{proposition}

\begin{remark}
Note that using \eqref{e-nonnegativity} and the form of the eigenfunctions (see \eqref{eigen}), it follows that for $\varepsilon$ small enough, the initial condition $f_0$ is non-negative in this statement.
\end{remark}

\begin{proof} Let us first prove the proposition for the case of $\mathcal{C}^2$ smooth potentials . Let $n \ge 3K+s$,  $m_0 > \frac d2$, and $m\ge m_0 + 2K$. In what follows, $W^{n,p}_x$ or $W^{n,p}_v$ denote the standard Sobolev spaces over $\mathbb{T}^d$ or $\mathbb{R}^d$, respectively. Norms without a subscript are taken over $\mathbb{T}^d \times \mathbb{R}^d$. By a view of $E[f]$, we easily obtain  
\begin{equation}\label{est-E} \|E[f]\|_{W_x^{n+1,\infty}} \le \| \nabla \Phi\|_{W^{1,\infty}_x} \| f \|_{W^{n,1}_x L^1_v} \le C_0 \|\langle v\rangle^{m_0}  f\|_{H^{n}_x L^2_v},\end{equation}
with which we have used $m_0 >\frac d2$ in the last estimate. We now give estimates on the approximate solution. By  construction, $g_1 = e^{\lambda_0 t} \hat g_1$, with sufficiently smooth eigenfunction $\hat g_1$ so that there holds
$$ \| \langle v \rangle^m g_1 (t) \|_{H^n} \le C_0 e^{\Re \lambda_0 t} .$$
Since $m \ge m_0$, this bound and \eqref{est-E} yield the same bound for  $E[g_1]$ in $W^{n+1,\infty}_x$ Sobolev space. 
Inductively, we shall prove that the profiles $g_k$ satisfy 
 \begin{align}\label{est-gk}
\| E[g_k](t)\|_{W^{n - 3k + 4, \infty}_x} + \| \langle v \rangle^{m - 2k+2 }g_k(t) \|_{H^{n-3k+3}} 
&\le  C_k e^{ k \Re \lambda_0 t}, \qquad k \ge 1.
\end{align}
It suffices to prove the bound for $g_k$. Indeed, the bound on $E[g_k]$ follows then from  \eqref{est-E}. By a view of \eqref{def-gk}, for all $k \geq 2$, $g_{k}$ satisfies the Duhamel's integral formula 
$$ g_{k} =  - \sum_{j = 1}^{k-1}  \int_0^t e^{L(t-s) }  E[g_{k-j}] (s) \cdot \nabla_v g_{j}(s) \; ds,$$
where $e^{Lt}$ denotes the semigroup of the linearized operator $(\partial_t - L)$. Using Lemma \ref{prop-eLs} of the appendix, we estimate for $\beta>0$
$$ 
\begin{aligned}
\| \langle v \rangle^{m - 2k } g_{k+1} \|_{H^{n-3k}} 
&\le  C_\beta \sum_{j = 1}^{k}  \int_0^t  e^{(\Re \lambda_0 + \beta)(t-s)} \| \langle v \rangle^{m - 2k + 2}E[g_{k-j+1}]  \cdot \nabla_v g_{j}\| _{H^{n-3k + 2}} \; ds
\\
&\le  C_\beta \sum_{j = 1}^{k}  \int_0^t  e^{(\Re \lambda_0 + \beta)(t-s)} \| E[g_{k-j+1}] \|_{W^{n-3k+2,\infty}_x}\| \langle v \rangle^{m - 2k + 2} g_{j}\| _{H^{n-3k + 3}} \; ds.
\end{aligned}$$ 
We take $\beta = \frac{\Re \lambda_0}{2}$. By induction, we have 
$$ 
\begin{aligned}
\| \langle v \rangle^{m - 2k } g_{k+1} \|_{H^{n-3k}} &\le  C_{k} \sum_{j = 1}^{k}  \int_0^t  e^{\frac 32\Re \lambda_0 (t-s)} e^{(k+1)\Re \lambda_0 s} \; ds \le C_{k+1} e^{(k+1) \Re \lambda_0 t}.
\end{aligned}$$ 
This proves \eqref{est-gk} for all $k\ge 1$. In addition, a direct computation yields 
$$ 
\begin{aligned}
\| \langle v \rangle^{m-2K} R_\mathrm{app}\|_{H^{n - 3K}} 
& \le  \sum_{1\le j,\ell \le K; j+\ell \ge K+1} \varepsilon^{j+k}  \| E[g_j] \|_{W^{n-3K, \infty}_x} \| \langle v \rangle^{m-3K} g_{\ell}\|_{H^{n-3K + 1}}
\\
& \le C_K \sum_{1\le j,\ell \le K; j+\ell \ge K+1} \varepsilon^{j+\ell }  e^{(j+\ell) \Re \lambda_0 t}
\\
&  \le C_{K}\Big(\varepsilon e^{ \Re \lambda_0 t}\Big)^{K+1},
\end{aligned} $$
as long as $\varepsilon e^{ \Re \lambda_0 t} \le \varepsilon_0$. This entails the $H^s$ bound for $R_\mathrm{app}$, up on recalling that $s = n - 3K$. It remains to show that $f_\mathrm{app}$ is a good approximation as claimed in the proposition. Indeed, let us denote $\tilde f = f -  f_\mathrm{app}$. Then, $\tilde f$ solves 
$$ \partial_t \tilde f + v \cdot \nabla_x \tilde f + E[f_\mathrm{app} + \tilde f] \cdot \nabla_v \tilde f + E [\tilde f] \cdot \nabla_v f_\mathrm{app} =  R_\mathrm{app},
$$
with zero initial data. An unique local solution $\tilde f$ exists in weighted $H^s$ function spaces for some short time thanks to the standard local theory for the Vlasov equation. We shall now show that the local solution can be continued up to the time $T_\varepsilon$, defined in the proposition. Indeed, standard energy estimates for arbitrary $s\ge 0$ and $m_0 > \frac d2$ yield 
$$
\begin{aligned}
 \frac{1}{2}\frac{d}{dt} \| \langle v \rangle^{m_0}\tilde f\|^2_{H^s} 
 &\le C \Big (1+ \| E[f_\mathrm{app}+ \tilde f]\|_{W^{s, \infty}_x}  \Big )  \| \langle v \rangle^{m_0}\tilde f\|_{H^s}^2 
 \\
 &\quad +  C\Big( \| E [\tilde f]\|_{W^{s,\infty}_x} \| \langle v \rangle^{m_0} f_\mathrm{app}\|_{H^s} +   \| \langle v \rangle^{m_0} R_\mathrm{app}\|_{H^s} \Big )  \| \langle v \rangle^{m_0}\tilde f\|_{H^s} .
  \end{aligned}$$  
For all times $t$ such that $\varepsilon e^{ \Re \lambda_0 t}\le \varepsilon_0$, using the bound on $f_\mathrm{app}$ and $R_\mathrm{app}$, together with \eqref{est-E}, we obtain  
$$
\begin{aligned}
\frac{d}{dt} \| \langle v \rangle^{m_0}\tilde f\|_{H^s} 
 &\le C_0 \Big (1+ C_K \varepsilon e^{\Re \lambda_0 t} +  \| \langle v \rangle^{m_0} \tilde f\|_{H^{s}} \Big )  \| \langle v \rangle^{m_0}\tilde f\|_{H^s} +  C_{K} \varepsilon^{K+1} e^{(K+1) \Re \lambda_0 t}.
  \end{aligned}$$  
Here, $C_0$ is some universal constant. We now take $K \ge 2C_0$ and $\varepsilon_0<1$ small enough so that $C_K \varepsilon_0\le 1$. 
Then, as long as $ \| \langle v \rangle^{m_0} \tilde f(t)\|_{H^{s}} \le C_K (\varepsilon e^{\Re \lambda_0 t})^K$, the Gronwall inequality yields 
$$
\begin{aligned}
\| \langle v \rangle^{m_0}\tilde f(t)\|_{H^s} 
 &\le C_{K} \varepsilon^{K+1} \int_0^t e^{C_0 (1 + C_K (\varepsilon e^{\Re \lambda_0 \tau })^K) (t - \tau)} e^{(K+1) \Re \lambda_0 \tau }\; d\tau
 \\ &\le C_{K} \varepsilon^{K+1} \int_0^t e^{2C_0 (t - \tau)} e^{(K+1) \Re \lambda_0 \tau }\; d\tau
 \\ &\le C_{K} \varepsilon^{K+1} e^{(K+1) \Re \lambda_0 t } \le C_K \varepsilon_0( \varepsilon e^{\Re \lambda_0 t})^K,
    \end{aligned}$$  
for all $t$ so that $\varepsilon e^{\Re \lambda_0 t} \le \varepsilon_0$. Since $\varepsilon_0<1$, a standard continuation argument shows that the solution can be continued and remains to satisfy the bound  $ \| \langle v \rangle^{m_0} \tilde f(t)\|_{H^{s}} \le C_K (\varepsilon e^{\Re \lambda_0 t})^K$ for all $t \in [0,T_\varepsilon]$. 

The proposition is proved for the case of $\mathcal{C}^2$ smooth potentials. The case of the singular Coulomb potential $\Phi$ is similar, up to a minor modification on the field estimate \eqref{est-E}. We skip the details.    
\end{proof}
 
 As a consequence of the very definition of $g_1$ and 
$f_{\mathrm{app}}$ and of Proposition~\ref{prop-Vapp} {, 
note also that we obtain the lower bound
 \begin{equation}\label{lower}
W_1 (f_\mathrm{app}(t), f_\infty )  \ge C_K \varepsilon e^{\Re \lambda_0 t}.
\end{equation}
 
We end this section by briefly explaining how to handle the case when all maximal eigenvalues are complex.
We assume here that $\Im \lambda_0 \neq 0$ and consider $h_1$ an eigenfunction. Writing ${h_1} = \Re  h_1 + i \Im h_1$, and assuming for instance that $\Re h_1 \neq 0$, we set  $g_1 = \Re (e^{\lambda_0 s} {h_1})$. The same construction and analysis can be performed, and this allows to prove Proposition~\ref{prop-Vapp}. Note however that the lower bound for $f_{\mathrm{app}}$ in~\eqref{lower} may not be achieved for all times, but at least for $t$ of the form $t= \frac{2\pi k}{\Im \lambda_0}$, $k \in \mathbb{N}^*$.  
 
%
%
%
%


%
%
%

\section{Instability analysis in $W_1$ distance}

%
We are now in position to prove Theorem~\ref{thm1}. We also start by assuming that  there is a real maximal eigenvalue $\lambda_0$.

Let $N \geq 1$ and $(X_{k,N}^0, V_{k,N}^0)_{1\leq k \leq N}$ be an initial configuration to be fixed later. 
Let 
$$ \mu_N(0): = \frac 1N \sum_{k=1}^N \delta_{(X_{k,N}^0, V_{k,N}^0)}$$ 
be the initial empirical measure and $\mu_N(t)$ the associated evolution through the Newton's dynamics.

We will rely on Dobrushin's or Hauray's Wasserstein stability estimate (see Theorems \ref{Dob} and \ref{Hau}). For any smooth $g(0)$, there is $C_2>0$, such that for all $N\geq 1$ and all $t \geq 0$,

\begin{equation}\label{was}  
W_1 (\mu_N(t), g(t)) \leq e^{C_2 t}  W_1 (\mu_N(0), g(0))  , 
\end{equation}
where $g(t)$ is the solution at time  $t$ of \eqref{Vlasov} with initial condition $g(0)$.
Using Proposition~\ref{prop-Vapp}, for $N$ large enough, $\varepsilon = \frac{1}{N^\alpha}$, with $\alpha$ to be fixed later, and for a large integer $K$ also to be chosen later, we obtain an associated distribution $f_{\mathrm{app}}$.

We write 
\begin{equation}\label{first}  
W_1 (\mu_N(t), f_\infty )  \ge W_1 (f_\mathrm{app}(t), f_\infty )  - W_1 (\mu_N(t), f_\mathrm{app})  
\end{equation}
in which we already know that 
\begin{equation}\label{expo}  W_1 (f_\mathrm{app} (t), f_\infty )  \ge  \frac{C}{N^\alpha} e^{\Re \lambda_0 t}, 
\end{equation}
see \eqref{lower}. Next, define $\tilde f_\mathrm{app}$ as the exact solution of the Vlasov equation \eqref{Vlasov} with the same initial data as $f_\mathrm{app}$, that is 
$$\tilde f_\mathrm{app} (0) = f_\mathrm{app} (0) = f_\infty +  \frac{1}{N^\alpha}  g_1.$$
We have then the bound
\begin{equation}\label{triangle}  
\begin{aligned}
W_1 (\mu_N(t), f_\mathrm{app}(t)) 
&\le W_1 (\mu_N(t), \tilde f_\mathrm{app}(t))   + W_1 (\tilde f_\mathrm{app}(t), f_\mathrm{app}(t)) 
\\
&\le e^{C_2 t} W_1 (\mu_N(0),  f_\mathrm{app} (0))  \qquad \mbox{using \eqref{was}}
\\&\qquad + C_0 \Big( \frac{1}{N^\alpha} e^{\Re \lambda_0 t}\Big)^K\qquad \mbox{using the bound \eqref{diff} in Proposition~\ref{prop-Vapp}}.
\end{aligned}
\end{equation}


Consider now the space  $(\mathbb{T}^d \times \mathbb{R}^d)^\mathbb{N}$ endowed by the measure $\mathbb{P}_{N} :=  f_\mathrm{app} (0)^{\otimes \mathbb{N}}$. We denote by $\mathbb{E}_{N}$ the expectation related to this measure. 

We consider an i.i.d. sequence $(X_{k,N}, V_{k,N})_{k \geq 1}$ of initial configurations distributed according to $f_\mathrm{app}(0)$. For $m \geq 1$, introduce the (random) empirical measure
\begin{equation}\label{def-mumN0} \mu_{m,N}(0) = \frac 1m \sum_{k=1}^m \delta_{(X_{k,N}^0, V_{k,N}^0)}.\end{equation}
We use a result of  Fournier and Guillin \cite[Theorem 1 (Moment estimates)]{FG}; we recall below a particular case which is sufficient for our needs.
\begin{theorem}
[Fournier-Guillin]
Let $\mu \in \mathcal{P}_1$. Let $q>2$. Assume that 
$$
M_q(\mu) := \int_{\mathbb{T}^d \times \mathbb{R}^d} |v|^q \mu(dx \, dv) <+\infty.
$$
Consider the space  $(\mathbb{T}^d \times \mathbb{R}^d)^\mathbb{N}$ endowed with the measure $\mu^{\otimes \mathbb{N}}$ and denote by $\mathbb{E}_\mu$ the expectation related to this measure. 
Consider an i.i.d. sequence $(X_{k}, V_{k})_{k \geq 1}$ distributed according to $\mu$. For $m \geq 1$, define the empirical measure
$ \mu_{m} = \frac 1m \sum_{k=1}^m \delta_{(X_{k}, V_{k})}$.

There exists $C>0$ depending only on $d,q$ such that for all $m \geq 1$,
$$
\begin{aligned}
\text{if } d= 1, \qquad \mathbb{E}_\mu (W_1 (\mu_m, \mu)) &\leq C M_q^{1/q}(\mu)
 \left(\frac{\log(1+m)}{m^{1/2}} + \frac{1}{m^{(q-1)/q}} \right),\\
\text{if } d\geq 2, \qquad \mathbb{E}_\mu (W_1 (\mu_m, \mu)) &\leq C M_q^{1/q}(\mu)
  \left(\frac{1}{m^{1/(2d)}} + \frac{1}{m^{(q-1)/q}} \right).
\end{aligned}
$$

\end{theorem}

Applying this result to $\mu_{m,N}(0)$, defined as in \eqref{def-mumN0}, and setting 
$$s:=\frac{1}{2}^- \delta_{d=1} + \frac{1}{2d} \delta_{d \geq 2},$$
where $\frac{1}{2}^-$ denotes any positive number less than $\frac{1}{2}$, 
we deduce that there is a $C>0$ such that for all $N \geq 1, m \geq 1$, 
\begin{align*}
\mathbb{E}_N [W_1 (\mu_{m,N}(0),  f_\mathrm{app} (0)) ] &\leq  \frac{C M_q^{1/q}(f_\mathrm{app}(0))}{m^s} \\
&\leq  \frac{C (f_\infty)}{m^s}.
\end{align*}
For $m=N$, this analysis reveals that with high probability, $W_1 (\mu_{N,N}(0),  f_\mathrm{app} (0)) \lesssim \frac{1}{N^s}$. In particular, there exists at least one configuration which we choose for $(X_{k,N}^0, V_{k,N}^0)_{1\leq k \leq N}$ so that, recalling the notation
$ \mu_N(0) = \frac 1N \sum_{k=1}^N \delta_{(X_{k,N}^0, V_{k,N}^0)}$,
we have
\begin{equation}
\label{W1}
W_1 (\mu_N(0), f_\mathrm{app} (0))  \leq \frac{C}{N^s}.
\end{equation}
We impose that the parameter $\alpha$ is such that $\alpha<s$, so that we have
$$
W_1 (\mu_N(0), f_\infty) \leq W_1 (\mu_N(0),  f_\mathrm{app} (0)) + W_1 ( f_\mathrm{app} (0), f_\infty) \lesssim \frac{1}{N^\alpha}.
$$
Then, combining \eqref{first}, \eqref{expo}, \eqref{triangle} and \eqref{W1}, we obtain
$$W_1 (\mu_N(t), f_\infty )  \ge  C \left(   \frac{1}{N^\alpha} e^{\Re \lambda_0 t}  - e^{C_2 t} \frac{1}{N^s} -  \Big(  \frac{1}{N^\alpha}e^{\Re \lambda_0 t}\Big)^K\right)  .$$  
By choosing $K$ sufficiently large and $\alpha$ sufficiently small  so that 
\begin{equation}\label{parameters} K\Re\lambda_0 \ge C_2, \qquad \alpha K \leq s\end{equation}
 the above becomes 
$$W_1 (\mu_N(t), f_\infty )  \ge   C \left( \frac{1}{N^\alpha} e^{\Re \lambda_0 t}  -  \Big(  \frac{1}{N^\alpha} e^{\Re \lambda_0 t}\Big)^K \right) \ge   \frac C2 \frac{1}{N^\alpha}     e^{\Re \lambda_0 t}$$
as long as $   \frac{1}{N^\alpha}  e^{\Re \lambda_0 t}$ remains smaller than one. We therefore obtain instability for $T_N = O (\log N)$, that is   
$$
\limsup_{N\to +\infty}  W_1 (\mu_{N}(T_N), f_\infty) >0.
$$
The proof of the theorem is complete for the case when $\lambda_0$ is real. The general case can be handled by recalling that the bound~\eqref{expo} holds for  times of the form $t= \frac{2\pi k}{\Im \lambda_0}$, $k \in \mathbb{N}^*$ and this allows to perform essentially the same analysis.

\appendix 

\section{Linear estimates}
In this section, we derive estimates on the semigroup $e^{Lt}$, with the linearized operator $L$ defined by 
\begin{equation}\label{def-L}
L f : =  - v \cdot \nabla_x f  -  E \cdot \nabla_v f_\infty, \qquad E (x) = - \iint_{\mathbb{T}^d\times \mathbb{R}^d} \nabla \Phi(x-y) f(y,v)\; dy dv,\end{equation}
for a fixed homogenous profile $f_\infty = f_\infty(v)$. The linear problem has been studied for the Coulomb potential; for instance, \cite{Degond, HKH} for one dimension and \cite{HKN} for higher dimensions. The proof presented in \cite{HKN} applies to the case of smooth potentials, and we shall reproduce it below for the sake of completeness. 
The estimates are sharp in terms of the growth in time; however we allow losses of derivatives and weights, but these losses are overcome in the scheme used in Proposition~\ref{prop-Vapp}.

\begin{lemma}[Sharp semigroup bounds]\label{prop-eLs} Let $\Phi$ be a $\mathcal{C}^2$ smooth potential and let $f_\infty(v)$ be a smooth unstable equilibrium of $L$ which decays sufficiently fast as $v \to \infty$, let $\lambda_0$ be a maximal unstable eigenvalue. Let $n\ge 0, m > \frac{d}{2} +1,$ and $h$ be in the $\langle v \rangle^{m+2}$-weighted Sobolev spaces $H^{n+2}$. Then, $f = e^{Lt} h$ is well-defined as the solution of the linear problem $(\partial_t - L) f =0$ with the initial data $h$. Furthermore, there holds 
 \begin{equation}
\label{bound-exp}  
\|  \langle v \rangle^m e^{Lt} h \|_{H^{n} } \le C_\beta e^{(\Re \lambda_0+\beta) t} \| \langle v \rangle^{m+2} h\|_{H^{n+2} }, \qquad \forall t\ge 0, \quad \forall \beta>0,
\end{equation}
 for some constant $C_\beta$ depending on $f_\infty$ and $ \beta$. 
 \end{lemma}
 
 \begin{proof} Let $n \ge 0$,  $m_0 > \frac d2$, and $m\ge m_0 + 1$. Consider the resolvent equation:
$$ (\lambda - L)f = h, \qquad   \, \Re \lambda \geq 0.$$
Standard $L^2$ energy estimates produce the bound
$$ 
\begin{aligned}
\Re \lambda \| \langle v\rangle^m f\|_{L^2} 
&\le \| \langle v\rangle^m \nabla_v f_\infty \|_{L^2_v } \| E\|_{L_x^\infty} + \| \langle v\rangle^m h\|_{L^2}\\
&\le C_0 \| \langle v\rangle^m \nabla_v f_\infty \|_{L^2_v} \| \langle v \rangle^{m_0} f\|_{L^2} + \| \langle v\rangle^m h\|_{L^2}
,\end{aligned}$$  
in which we have used the estimate \eqref{est-E} on $E$ in terms of $f$. Recalling $m\ge m_0+1$, we thus deduce the following weighted $L^2$ resolvent bound
$$
\| \langle v\rangle^m (\lambda - L)^{-1} h\|_{L^2} \le \frac{1}{\Re \lambda - \gamma_{0,m} } \| \langle v\rangle^m h\|_{L^2}
$$
 for all $\Re \lambda > \gamma_{0,m}:= C_0 \| \langle v\rangle^m \nabla_v f_\infty \|_{L^2_v}$. Similarly,  estimates for higher order derivatives are obtained, since we observe that we have for $\alpha, \beta \in \mathbb{N}^d$,
$$
\begin{aligned}
 \lambda \partial_v^\beta \partial_x^\alpha f + v \cdot \nabla_x \partial_v^\beta \partial_x^\alpha f  - \nabla_v \partial_v^\beta f_\infty \cdot  \partial_x^\alpha E  + [\partial_v^\beta, v\cdot \nabla_x] \partial_x^\alpha f &= \partial_v^\beta \partial_x^\alpha h
\end{aligned}
 $$
in which 
$[\partial_v^\beta, v\cdot \nabla_x] = \partial_v^\beta (v\cdot \nabla_x ) - v\cdot \nabla_x \partial_v^\beta$. Combining with the estimate \eqref{est-E}, we get
$$
\begin{aligned}
\Re \lambda \| \langle v\rangle^m \partial_x^\alpha f \|_{L^2} 
&\le \| \langle v\rangle^m\nabla_v  f_\infty \cdot  \partial_x^\alpha E \|_{L^2} +  \| \langle v\rangle^m \partial_x^\alpha h \|_{L^2} 
\\& \le C_0 \| \langle v\rangle^m\nabla_v f_\infty\|_{L^2_v}  \|\langle v\rangle^{m_0} f\|_{H^{|\alpha|-1 }_x L^2_v}  +  \| \langle v\rangle^m \partial_x^\alpha h \|_{L^2}.
\end{aligned}
 $$
Likewise, we have
$$
\begin{aligned}
\Re \lambda \| \langle v\rangle^m \partial_v^\beta \partial_x^\alpha f \|_{L^2} 
&\le \| \langle v\rangle^m\nabla_v \partial_v^\beta f_\infty \cdot  \partial_x^\alpha E\|_{L^2}  + \| \langle v\rangle^m[\partial_v^\beta, v\cdot \nabla_x] \partial_x^\alpha f\|_{L^2} + \| \langle v\rangle^m \partial_v^\beta \partial_x^\alpha h \|_{L^2} 
\\
&\le C_0 \| \langle v\rangle^m\nabla_v \partial_v^\beta f_\infty\|_{L^2_v}  \| \langle v\rangle^{m_0}  f\|_{H^{|\alpha|-1}_x L^2_v}  \\
&\qquad \qquad+ C_0 \| \langle v\rangle^m f \|_{ H_x^{|\alpha|+1}  H_v^{|\beta|-1}} + \| \langle v\rangle^m \partial_v^\beta \partial_x^\alpha h\|_{L^2}.
\end{aligned}
 $$
We therefore obtain by induction
$$\begin{aligned}
\Re \lambda \| \langle v\rangle^m \partial_v^\beta \partial_x^\alpha f \|_{L^2} \le  C'_0 \| \langle v\rangle^m f\|_{H^{|\alpha|+|\beta|}_x L^2_v} + C'_0 \| \langle v\rangle^m h\|_{H^{n}} ,
\end{aligned}
$$
for all multi-indices $\alpha, \beta$ such that $|\alpha| + |\beta| \leq n$. By induction, this proves that there exists  $\gamma_{n,m}, C_{n,m}>0$ such that
\begin{equation}\label{high-bound}
\begin{aligned}
\Re \lambda \| \langle v \rangle^m f \|_{H^{n}} 
&\le \gamma_{n,m} \| \langle v \rangle^{m_0} f\|_{L^2}   + C_{n,m} \| \langle v\rangle^m h\|_{H^{n}} , 
\end{aligned}
 \end{equation}for all $n\ge 0$ and $\Re \lambda >0$. In particular, this proves that
$$\| \langle v \rangle^m  (\lambda - L)^{-1} h\|_{H^{n}} \le \frac{C_{n,m}}{\Re \lambda - \gamma_{n,m}} \| \langle v \rangle^m h\|_{H^{n}},$$
for some positive constant $C_{\gamma_{n,m}}$, and for all $\lambda \in \mathbb{C}$ so that $\Re \lambda >\gamma_{n,m}$. The classical Hille-Yosida theorem then asserts that $L$ generates a continuous semigroup $e^{Lt}$ on the Banach space $H^{n}$ with weights; see, for instance, \cite{Pazy} or \cite[Appendix A]{Zum}. We have furthermore the representation formula
\begin{equation}
\label{eLs}
e^{Lt} h= \text{P.V. } \frac{1}{2\pi i} \int_{\gamma - i\infty}^{\gamma + i \infty} e^{\lambda t} (\lambda - L)^{-1} h \; d\lambda \end{equation}
for any $\gamma > \gamma_{n,m}$, where $\text{P.V. }$ denotes the Cauchy principal value.

Next, by assumption, $\lambda_0$ is an unstable eigenvalue with maximal real part, and the resolvent operator $(\lambda - L)^{-1}$ is in fact a well-defined and bounded operator on the weighted space $H^{n}$ for all $\lambda$ so that $\Re \lambda > \Re \lambda_0$. Let $\beta>0$. By  Cauchy's theorem, we can take $\gamma = \Re \lambda_0 + \beta$ in the representation \eqref{eLs}. Since the resolvent operator is bounded, we obtain at once
\begin{equation}\label{bounded-l} \Big\| \frac{\langle v \rangle^m}{2\pi i} \int_{\gamma - iM}^{\gamma + i M} e^{\lambda t} (\lambda - L)^{-1} h \; d\lambda \Big\|_{H^{n}} \le C_{\beta, M} e^{(\Re \lambda_0 + \beta) t} \| \langle v \rangle^m h\|_{H^{n}},\end{equation}
for any large but fixed constant $M$. For what concerns large values of $\Im \lambda$, we observe directly from the equation $\lambda f = L f + h$, that one has
$$ |\lambda| \| \langle v \rangle^m f \|_{H^{n}} \le \| \langle v \rangle^m (L f + h) \|_{H^{n}} \le C ( \| \langle v \rangle^{m+1} \nabla_xf\|_{H^{n}} + \| E\|_{W^{n,\infty}_x})+ \| \langle v \rangle^m h\|_{H^{n}}.$$
Using \eqref{high-bound}, we get, for some $C'_{n+1,m+1}>0$,
$$
 |\lambda|   \| \langle v \rangle^m (\lambda - L)^{-1} h\|_{H^{n}} \le {C'_{n+1,m+1}} \| \langle v \rangle^{m_0} f \|_{L^2} + C'_{n+1,m+1} \| \langle v \rangle^{m+1}h \|_{H^{n+1}}.
$$
We take $ \Re \lambda = \gamma$, and take 
$$|\Im \lambda|>\frac{2}{3}C'_{n+1,m+1}.$$
We end up with 
\begin{equation}\label{large-l} \| \langle v \rangle^m (\lambda - L)^{-1} h\|_{H^{n}} \le \frac{C_\beta}{|\Im \lambda|} \| \langle v \rangle^{m+1} h\|_{H^{n+1}}.\end{equation}
Finally, in order to bound the integral for large $|\Im \lambda|$, we  write $$(\lambda - L)^{-1} h = \frac1\lambda (\lambda - L)^{-1} Lh + \frac h \lambda.$$
As a consequence, for $\gamma = \Re \lambda_0 + \beta$, we get
$$ 
\begin{aligned}
\text{P.V. } &\frac{1}{2\pi i} \int_{\{|\Im \lambda |\ge M\}}e^{\lambda t} (\lambda - L)^{-1} h \; d\lambda  
\\&= \text{P.V. }\frac{1}{2\pi i} \int_{\{|\Im \lambda|\ge M\}} e^{\lambda t} (\lambda - L)^{-1} \frac{Lh}{\lambda} \; d\lambda 
+ \text{P.V. } \frac{1}{2\pi i} \Big[ \int_{\gamma - i\infty}^{\gamma + i \infty}  - \int_{\{|\Im \lambda|< M\}} \Big] e^{\lambda t} \frac{h}{\lambda} \; d\lambda 
\end{aligned}$$
in which the second integral on the right-hand side is equal to $h$, whereas the last integral is bounded by $C_0e^{\gamma t}h$. We consider $M\geq \frac{2}{3}C'_{n+1,m+1}$ so that the bound \eqref{large-l} holds. We deduce
$$
\begin{aligned}
 \Big \| \langle v \rangle^m \int_{\{|\Im \lambda|\ge M\}} e^{\lambda t} (\lambda - L)^{-1} \frac{Lh}{\lambda} \; d\lambda \Big\|_{H^{n}} 
&\le C_{\beta, M} e^{\gamma t} \| \langle v \rangle^{m+1} L h\|_{H^{n+1}}  \int_{\{|\Im \lambda|\ge M\}} |\Im \lambda|^{-2} \; d\Im \lambda 
\\
&\le
C_{\beta, M} e^{\gamma t} \| \langle v \rangle^{m+2} h\|_{H^{n+2}}.
\end{aligned}$$
Combining this estimate with \eqref{bounded-l} and \eqref{eLs}, we conclude the proof of the lemma.  
 \end{proof}

\bigskip

\noindent {\bf Acknowledgments.} The first author thanks Maxime Hauray and Fr\'ed\'eric Rousset for stimulating discussions.
We are also grateful to Maxime Hauray for explaining to us how to build $N$ particles global flows in the one-dimensional coulombian case.

We finally thank the anonymous referees for several insightful comments and suggestions about this work.

\bibliographystyle{plain}
\bibliography{meanfield}
 
\end{document}